\pdfoutput=1
\documentclass[english]{jnsao}
\usepackage[utf8]{inputenc}

\usepackage[nameinlink,capitalize]{cleveref}

\acknowledgements{
	The author would like to thank Gerd Wachsmuth 
	for making the author aware of the conjecture in the Diploma thesis \cite{Schinabeck2009:1}.
	This work is supported by the DFG Grant 
	\emph{Bilevel Optimal Control: Theory, Algorithms, and Applications}
	(Grant No.\ WA 3636/4-2) 
	within the Priority Program SPP 1962 
	(Non-smooth and Complementarity-based Distributed Parameter Systems: 
	Simulation and Hierarchical Optimization).
}

\manuscriptsubmitted{2020-11-13}
\manuscriptaccepted{2021-10-20}
\manuscriptvolume{2}
\manuscriptnumber{6903}
\manuscriptyear{2021}
\manuscriptdoi{10.46298/jnsao-2021-6903}

\manuscriptlicense{CC-BY-SA 4.0}

\usepackage{oc_others}
\usepackage{oc_sets}
\usepackage{oc_mathcals}

\newenvironment{msc}%
{\par\quotation\noindent{\sectfont MSC (2020):}\ }{\par}

\newcommand{\mscLink}[1]{\href{https://www.ams.org/mathscinet/msc/msc2020.html?t=#1}{#1}}

\newenvironment{keywords}%
{\par\quotation\noindent{\sectfont Keywords:}\ }{\par}

\usepackage{enumitem}
\setlist[enumerate]{label=(\alph*)}

\renewcommand\phi\varphi
\newcommand\eps\varepsilon
\newcommand\tmpcclin{\TT_{\text{MPCC}}^{\text{lin}}}
\newcommand\tnlplin[1]{\TT_{\text{NLP}(#1)}^{\text{lin}}}

\definecolor{todocolor}{rgb}{1.0,0.3,0.3}

\allowdisplaybreaks

\title{%
A new elementary proof for M-stationarity under MPCC-GCQ
for mathematical programs with complementarity constraints}
\shorttitle{A new elementary proof for M-stationarity}
\author{Felix Harder
	\thanks{
		Institute of Mathematics, Chair of Optimal Control.
		Brandenburgische Technische Universität Cottbus--Senftenberg, Germany.
		\email{harder@b-tu.de}.
	}
}
\shortauthor{Harder}

\date{2021-10-14}

\begin{document}
\maketitle
\begin{abstract}
	It is known in the literature that local 
	minimizers of mathematical programs with complementarity constraints
	(MPCCs) are so-called M-stationary points,
	if a weak MPCC-tailored Guignard constraint qualification 
	(called MPCC-GCQ) holds.
	In this paper we present a new elementary proof for this result.
	Our proof is significantly simpler than existing proofs
	and does not rely on deeper technical theory
	such as calculus rules for limiting normal cones.
	A crucial ingredient is a proof of a
	(to the best of our knowledge previously open) conjecture,
	which was formulated in a Diploma thesis by Schinabeck.
\end{abstract}

\begin{keywords}
	Mathematical program with complementarity constraints, 
	Necessary optimality conditions,
	M-stationarity,
	Guignard constraint qualification
\end{keywords}

\begin{msc}
	\mscLink{90C33}, 
	\mscLink{90C30} 
\end{msc}


\section{Introduction}

We consider mathematical programs with complementarity constraints,
or MPCCs for short,
which are nonlinear optimization problems of the form
\begin{equation}
	\label{eq:mpcc}
	\tag{MPCC}
	\begin{aligned}
		\min_{x\in\R^n} \quad& f(x)
		\\
		\text{s.t.}\quad&
		\begin{aligned}[t]
			g(x)&\leq 0,
			& h(x)&=0,
			\\
			G(x) &\geq0,
			& H(x) &\geq 0,
			& G(x)^\top H(x) &=0.
		\end{aligned}
	\end{aligned}
\end{equation}
Here, $f:\R^n\to\R$, $g:\R^n\to\R^l$, $h:\R^n\to\R^m$,
$G,H:\R^n\to\R^p$
are differentiable functions.

MPCCs have been studied extensively in the literature,
both from a numerical and a theoretical perspective.
Various problem-tailored stationarity conditions and constraint qualifications
have been developed.
One (necessary, first-order) 
stationarity condition for MPCCs is the so-called strong stationarity
(which is equivalent to the KKT conditions).
However, the strong stationarity condition is often too strong,
as there are examples of MPCCs where the data
is linear, but the unique minimizer is not strongly stationary,
see \cite[Example~3]{ScheelScholtes2000}.

In this article we will focus on M-stationarity (see \cref{def:mstat}),
which is another prominent stationarity condition
for MPCCs in the literature.
We mention that there are other stationarity conditions between
M- and strong stationarity, such as
extended M-stationarity, strong M-stationarity,
$\QQ_M$-stationarity, and linearized M-stationarity,
see \cite{Gfrerer2014,BenkoGfrerer2017,Gfrerer2018}.
However, these conditions are used less frequently than
M-stationarity and are more complicated to formulate.

As shown in \cite{FlegelKanzowOutrata2006},
M-stationarity holds for local minimizers of \eqref{eq:mpcc}
under MPCC-GCQ, see also \cite{FlegelKanzow2006:1}.
MPCC-GCQ, defined in \cref{def:mpcc_gcq},
is a relatively weak constraint qualification.
These proofs for M-stationarity under MPCC-GCQ rely on
advanced techniques from variational analysis such as
the concept of the so-called limiting normal cone.
In particular, calculus rules for limiting normal cones
are used, which are based on deeper technical theory
and require to verify the calmness of certain set-valued mappings.


We mention that there is also a weaker constraint qualification
in \cite[(6)]{GuoLin2012}, which is the weakest constraint
qualification under which local minimizers are M-stationary.
An advantage is that the proof of M-stationarity under this 
constraint qualification is very simple,
but on the other hand the proof that MPCC-GCQ implies this constraint qualification
(or that this constraint qualification is satisfied for
MPCCs with affine data) relies on deeper results
from variational analysis.
Since MPCC-GCQ and MPCC-ACQ (which trivially implies MPCC-GCQ)
are used more frequently, we focus on MPCC-GCQ in this article.

In this paper we want to present
a new proof for 
the result from \cite{FlegelKanzowOutrata2006}
that M-stationarity holds for local minimizers under MPCC-GCQ.
Our new proof does not rely on advanced theory 
such as the properties of limiting normal cones.
Our proof is significantly simpler than any existing proofs
that we are aware of.
%
The main ingredient of the new proof
is a result which can be found in
\cref{lem:schinabeck}. 
This result was already conjectured in \cite[Section~4.4.2]{Schinabeck2009:1},
and, to the best of our knowledge there has not been a proof of this
conjecture so far.
With the knowledge that \cref{lem:schinabeck} holds,
the rest of the proof of
M-stationarity under MPCC-GCQ will not be particularly surprising
for readers familiar with the implications of MPCC-GCQ.
For the convenience of the reader we give a self-contained presentation, 
which only requires familiarity with the basic theory of nonlinear optimization.
Since the constraint qualification in \cite[(6)]{GuoLin2012}
is the weakest constraint qualification under which
local minimizers are M-stationary, the new proof
also leads to a completely elementary proof that
MPCC-GCQ implies \cite[(6)]{GuoLin2012}.


The structure of this paper is as follows:
In \cref{sec:definitions}, we introduce classical definitions related to MPCCs.
Then we use MPCC-GCQ to construct various multipliers
that satisfy some auxiliary stationarity systems, see \cref{lem:astat}.
Afterwards, these multipliers are combined into
a multiplier which is M-stationary with the help
of \cref{lem:schinabeck}.
The main result is then stated in \cref{thm:mstat}.
Finally, we discuss conclusions and perspectives in
\cref{sec:conclusion}.

\section{Definitions}
\label{sec:definitions}

It will be convenient to work with the index sets
\begin{align*}
	I^l &:=\set{1,\ldots,l},\qquad
	I^m :=\set{1,\ldots,m},\qquad
	I^p :=\set{1,\ldots,p},
	\\
	I^g(\bar x)&:=\{i\in I^l\,|\,g_i(\bar x)=0\},\\
	I^{+0}(\bar x)&:=\{i\in I^p\,|\,G_i(\bar x)>0\,\land\,H_i(\bar x)=0\},\\
	I^{0+}(\bar x)&:=\{i\in I^p\,|\,G_i(\bar x)=0\,\land\,H_i(\bar x)>0\},\\
	I^{00}(\bar x)&:=\{i\in I^p\,|\,G_i(\bar x)=0\,\land\,H_i(\bar x)=0\},
\end{align*}
where $\bar x\in\R^n$ is a feasible point of \eqref{eq:mpcc}.
Note that $I^{+0}(\bar x)$, $I^{0+}(\bar x)$, $I^{00}(\bar x)$ 
form a partition of $I^p$.
We continue with the definition of M-, A- and S-stationarity.
%
\begin{definition}
	\label{def:mstat}
	Let $\bar x\in\R^n$ be a feasible point of \eqref{eq:mpcc}.
	We call $\bar x$ an \emph{M-stationary} point of \eqref{eq:mpcc}
	if there exist multipliers
	$\bar\lambda\in\R^l$, $\bar\eta\in\R^m$,
	$\bar\mu,\bar\nu\in\R^p$ with
	\begin{subequations}
		\begin{align}
			\label{eq:wstat_x}
			&&&&\mathllap{
			\nabla f(\bar x) 
			+ \sum_{i\in I^l} \bar\lambda_i\nabla g_i(\bar x)
			+ \sum_{i\in I^m} \bar\eta_i\nabla h_i(\bar x)
			- \sum_{i\in I^p} \paren[\big]{\bar\mu_i\nabla G_i(\bar x)
			+ \bar\nu_i\nabla H_i(\bar x)}
		}
			&=0, \\
			&\hspace{9em}&&\forall i\in I^g(\bar x):
			&
			\bar\lambda_i &\geq0,
			\\
			&&&\forall i\in I^l\setminus I^g(\bar x):
			&
			\bar\lambda_i &=0,
			\\
			&&&\forall i\in I^{+0}(\bar x):
			&
			\bar\mu_i  &=0,
			\\
			&&&\forall i\in I^{0+}(\bar x):
			&
			\bar\nu_i  &=0,
			\label{eq:wstat_0+}
			\\
			&&&\forall i\in I^{00}(\bar x):
			&
			(\bar\mu_i>0 \land \bar\nu_i>0)
			\lor \bar\mu_i\bar\nu_i &=0.
			\label{eq:mstat_00}
		\end{align}
	\end{subequations}
	If the multipliers $\bar\lambda,\bar\eta,\bar\mu,\bar\nu$ only satisfy 
	\eqref{eq:wstat_x}--\eqref{eq:wstat_0+}
	and $\bar\mu_i\geq0\lor\bar\nu_i\geq0$ 
	holds for all $i\in I^{00}(\bar x)$,
	then $\bar x$ is called an \emph{A-stationary} point of \eqref{eq:mpcc}.
	If, additionally, $\bar\mu_i\geq0\land\bar\nu_i\geq0$
	holds for all $i\in I^{00}(\bar x)$, then
	$\bar x$ is called a \emph{strongly stationary} or \emph{S-stationary}
	point of \eqref{eq:mpcc}.
\end{definition}
These stationarity conditions can be found in
\cite[Definitions~2.5--2.7]{Ye2005}.
Other known stationarity conditions
B-, W- and C-stationarity, see \cite[Definitions~2.2--2.4]{Ye2005}.

In preparation for the definition of MPCC-GCQ
we define some cones.
\begin{definition}
	\label{def:general}
	Let $\bar x\in\R^n$ be a feasible point of \eqref{eq:mpcc}.
	\begin{enumerate}
		\item
			\label{item:tangent_cone}
			We define the \emph{tangent cone} of \eqref{eq:mpcc}
			at $\bar x$ via
			\begin{equation*}
				\TT(\bar x):=\set*{ d\in\R^n
					\given
					\begin{aligned}
						&\exists \set{x_k}_{k\in\N}\subset F,\,
						\exists \set{t_k}_{k\in\N}\subset(0,\infty):
						\\ &\qquad
						x_k\to \bar x,\; t_k\downto 0,\;
						t_k^{-1}(x_k-\bar x)\to d
					\end{aligned}
				},
			\end{equation*}
			where $F\subset \R^n$ denotes the feasible set of \eqref{eq:mpcc}.
		\item
			\label{item:mpcclin}
			We define the \emph{MPCC-linearized tangent cone}
			$\tmpcclin(\bar x)\subset\R^n$ at $\bar x$ via
			\begin{equation*}
				\tmpcclin(\bar x):=
				\set*{ d\in\R^n
					\given
					\begin{aligned}
						\nabla g_i(\bar x)^\top d &\leq 0
						\qquad \forall i\in I^g(\bar x),
						\\
						\nabla h_i(\bar x)^\top d &= 0
						\qquad \forall i\in I^m,
						\\
						\nabla G_i(\bar x)^\top d &= 0
						\qquad \forall i\in I^{0+}(\bar x),
						\\
						\nabla H_i(\bar x)^\top d &= 0
						\qquad \forall i\in I^{+0}(\bar x),
						\\
						\nabla G_i(\bar x)^\top d &\geq 0
						\qquad \forall i\in I^{00}(\bar x),
						\\
						\nabla H_i(\bar x)^\top d &\geq 0
						\qquad \forall i\in I^{00}(\bar x),
						\\
						(\nabla G_i(\bar x)^\top d )(\nabla H_i(\bar x)^\top d)&= 0
						\qquad \forall i\in I^{00}(\bar x)
					\end{aligned}
				}.
			\end{equation*}
	\end{enumerate}
\end{definition}
%
Note that in general
$\TT(\bar x)$ and $\tmpcclin(\bar x)$ are
nonconvex cones.

We also recall that the polar cone $C\polar$ of a set $C\subset\R^n$
is defined via
\begin{equation*}
	C\polar:=\set{d\in\R^n\given d^\top y\leq0 \quad\forall y\in C}.
\end{equation*}
Now we are ready to give the definition of
MPCC-GCQ, which can also be found in
\cite[(41)]{FlegelKanzowOutrata2006},
where it is called MPEC-GCQ.
\begin{definition}
	\label{def:mpcc_gcq}
	Let $\bar x\in\R^n$ be a feasible point of \eqref{eq:mpcc}.
	We say that $\bar x$ satisfies the
	\emph{MPCC-tailored Guignard constraint qualification},
	or \emph{MPCC-GCQ}, if
	\begin{equation*}
		\TT(\bar x)\polar = \tmpcclin(\bar x)\polar
	\end{equation*}
	holds.
	Additionally, if $\TT(\bar x)=\tmpcclin(\bar x)$ holds
	then we say that $\bar x$ satisfies \emph{MPCC-ACQ}.
\end{definition}
Clearly, MPCC-ACQ implies MPCC-GCQ.
We mention that there are also other stronger constraint qualifications
(such as MPCC-MFCQ if $g$, $h$, $G$, $H$ are continuously differentiable,
defined in \cite[Definition~2.1]{FlegelKanzow2005:3})
which imply MPCC-ACQ or MPCC-GCQ and are sometimes easier to verify,
see e.g.\ \cite[Theorem~3.2]{Ye2005}.
In particular, we emphasize that MPCC-GCQ (and MPCC-ACQ)
are satisfied at every feasible point of \eqref{eq:mpcc}
if the functions $g$, $h$, $G$, $H$ are affine,
see \cite[Theorem~3.2]{FlegelKanzow2005:3}.

%
%

\section{M-stationarity under MPCC-GCQ}
\label{sec:mstat}

%
We start with a \lcnamecref{lem:astat}
that generates several multipliers which satisfy
a slightly stronger stationarity condition than A-stationarity.
The result can also be obtained from the proof of
\cite[Theorem~3.4]{FlegelKanzow2005:3},
with the minor difference that we only require MPCC-GCQ
and not MPCC-ACQ.
%
\begin{proposition}
	\label{lem:astat}
	Let $\bar x\in\R^n$ be a local minimizer of \eqref{eq:mpcc}
	that satisfies MPCC-GCQ
	and let $\alpha\in\set{1,2}^p$ be given.
	Then there exist multipliers
	$\lambda^\alpha\in\R^l$, $\eta^\alpha\in\R^m$,
	$\mu^\alpha,\nu^\alpha\in\R^p$ with
	\begin{subequations}
		\label{eq:astat}
		\begin{align}
			\label{eq:astat_x}
			&&&&\mathllap{
			\nabla f(\bar x) 
			+ \sum_{i\in I^l} \lambda_i^\alpha\nabla g_i(\bar x)
			+ \sum_{i\in I^m} \eta_i^\alpha\nabla h_i(\bar x)
			- \sum_{i\in I^p}\paren[\big]{ \mu_i^\alpha\nabla G_i(\bar x)
			+  \nu_i^\alpha\nabla H_i(\bar x)}
		}
			&=0, \\
			&\hspace{16em}&&\forall i\in I^g(\bar x):
			&
			\lambda_i^\alpha &\geq0,
			\\
			&&&\forall i\in I^l\setminus I^g(\bar x):
			&
			\lambda_i^\alpha &=0,
			\\
			&&&\forall i\in I^{+0}(\bar x):
			&
			\mu_i^\alpha  &=0,
			\\
			&&&\forall i\in I^{0+}(\bar x):
			&
			\nu_i^\alpha  &=0,
			\label{eq:astat_0+}
			\\
			&&&\forall i\in I^{00}(\bar x),\; \alpha_i=1:
			&
			\quad\mu_i^\alpha &\geq0,
			\\
			&&&\forall i\in I^{00}(\bar x),\; \alpha_i=2:
			&
			\quad\nu_i^\alpha &\geq0.
		\end{align}
	\end{subequations}
\end{proposition}
\begin{proof}
	From \cref{def:general}~\ref{item:tangent_cone}
	and the fact that $\bar x$ is a local minimizer of
	\eqref{eq:mpcc} it can be concluded that
	the condition
	\begin{equation*}
		\nabla f(\bar x)^\top d\geq0
		\qquad\forall\, d\in\TT(\bar x)
	\end{equation*}
	is satisfied.
	Using polar cones and MPCC-GCQ, we obtain
	\begin{equation*}
		-\nabla f(\bar x)\in \TT(\bar x)\polar
		= \tmpcclin(\bar x)\polar.
	\end{equation*}
	Furthermore, we define
	the cone $\tnlplin{\alpha}(\bar x)\subset \tmpcclin(\bar x)$ 
	via
	\begin{equation*}
		\tnlplin{\alpha}(\bar x):=
		\set*{d\in\tmpcclin(\bar x)\given
			\begin{aligned}
				(\alpha_i=1\Rightarrow \nabla H_i(\bar x)^\top d &=0) 
				\;&&\forall i\in I^{00}(\bar x),
				\\
				(\alpha_i=2\Rightarrow\nabla G_i(\bar x)^\top d &=0) 
				\;&&\forall i\in I^{00}(\bar x)
			\end{aligned}
		}.
	\end{equation*}
	Using the definition of polar cones, the inclusion
	$\tmpcclin(\bar x)\polar\subset\tnlplin{\alpha}(\bar x)\polar$
	follows from
	$\tnlplin{\alpha}(\bar x)\subset\tmpcclin(\bar x)$.
	In particular, we have
	$-\nabla f(\bar x)\in\tnlplin{\alpha}(\bar x)\polar$.
	Note that the condition
	$(\nabla G_i(\bar x)^\top d )(\nabla H_i(\bar x)^\top d)= 0
	\;\forall i\in I^{00}(\bar x)$
	from $\tmpcclin(\bar x)$ is redundant in $\tnlplin{\alpha}(\bar x)$,
	and therefore $\tnlplin{\alpha}(\bar x)$ is a convex
	and polyhedral cone (unlike $\tmpcclin(\bar x)$).
	Thus, one can calculate its polar cone 
	(e.g.\ using Farkas' Lemma),
	which results in
	\begin{equation*}
		\tnlplin{\alpha}(\bar x)\polar
		= \set*{
			\begin{aligned}
				&\sum_{i\in I^g(\bar x)} \lambda_i^\alpha \nabla g_i(\bar x)
				+ \sum_{i\in I^m} \eta_i^\alpha \nabla h_i(\bar x)
				\\
				&\quad -
				\sum_{i\in I^{0+}(\bar x)\cup I^{00}(\bar x)} \mu_i^\alpha \nabla G_i(\bar x)
				\\ &\quad -
				\sum_{i\in I^{+0}(\bar x)\cup I^{00}(\bar x)} \nu_i^\alpha \nabla H_i(\bar x)
			\end{aligned}
			\given
			\begin{aligned}
				\lambda_i^\alpha &\geq0,\quad i\in I^g(\bar x),
				\\
				\eta_i^\alpha &\in\R,\quad i\in I^m,
				\\
				\mu_i^\alpha &\in\R,\quad i\in I^{0+}(\bar x)\cup I^{00}(\bar x),
				\\
				\nu_i^\alpha &\in\R,\quad i\in I^{+0}(\bar x)\cup I^{00}(\bar x),
				\\
				\mu_i^\alpha &\geq0,\quad\text{if}\;\alpha_i=1,\,i\in I^{00}(\bar x),
				\\
				\nu_i^\alpha &\geq0,\quad\text{if}\;\alpha_i=2,\,i\in I^{00}(\bar x)
			\end{aligned}
		}.
	\end{equation*}
	Then the result follows from
	$-\nabla f(\bar x)\in\tnlplin{\alpha}(\bar x)\polar$
	by setting the remaining components of the multipliers
	(i.e.\ $\lambda_i^\alpha$ for $i\in I^p\setminus I^g(\bar x)$,
		$\mu_i^\alpha$ for $i\in I^{+0}(\bar x)$, 
	$\nu_i^\alpha$ for $i\in I^{0+}(\bar x)$)
	to zero.
\end{proof}
We mention that if $\bar x$ satisfies \eqref{eq:astat}
for some $\alpha\in\set{1,2}^p$ and suitable multipliers,
then $\bar x$ is an A-stationary point of \eqref{eq:mpcc}.
However, the statement of \cref{lem:astat} is stronger than A-stationarity,
namely for each index in $I^{00}(\bar x)$ we can choose whether 
$\mu_i^\alpha$ or $\nu_i^\alpha$ is nonnegative.
Note that \eqref{eq:wstat_x}--\eqref{eq:wstat_0+}
are already satisfied by all $2^p$ possible choices for the multipliers
and any convex combination of these.
Thus, the question naturally arises whether
a convex combination of these multipliers can be found that
also satisfies \eqref{eq:mstat_00}.
As the next result shows, this is indeed possible.
The following \lcnamecref{lem:schinabeck} 
was already stated as a conjecture in 
\cite[Section~4.4.2]{Schinabeck2009:1}.
To the best of our knowledge, 
this conjecture has not been proven before.
%
\begin{lemma}
	\label{lem:schinabeck}
	Let $\hat I\subset I^p$ be an index set.
	For all $\alpha\in\set{1,2}^p$,
	let points $(\mu^\alpha,\nu^\alpha)\in A^\alpha$ be given, 
	where the set $A^\alpha$ is described via
	\begin{equation*}
		A^\alpha :=
		\set{(\mu,\nu)\in\R^{2p}\given 
			\mu_i\geq0\;\text{if}\;\alpha_i=1,\,
			\nu_i\geq0\;\text{if}\;\alpha_i=2
			\quad\forall i\in\hat I
		}.
	\end{equation*}
	Then there exists a point $(\bar\mu,\bar\nu)$ in the set
	\begin{equation*}
		B :=\conv\set[\big]{(\mu^\alpha,\nu^\alpha)\given \alpha\in\set{1,2}^p}
		\subset \R^{2p}
	\end{equation*}
	of convex combinations of these points, such that
	for all $i\in \hat I$ we have the condition
	\begin{equation}
		\label{eq:m_stat}
		(\bar\mu_i>0\land\bar\nu_i>0) \lor \bar\mu_i\bar\nu_i=0.
	\end{equation}
\end{lemma}
\begin{proof}
	Let us choose points 
	$(\hat\mu^\alpha,\hat\nu^\alpha)\in B\cap A^\alpha$,
	a vector $\beta\in\set{1,2}^p$,
	and points $\bar\mu,\bar\nu\in\R^p$
	that satisfy
	\begin{align}
		\label{eq:argmin_choice}
		(\hat\mu^\alpha,\hat\nu^\alpha) &\in
		\argmin_{(\mu,\nu)\in B\cap A^\alpha} \,\norm{(\mu,\nu)}_2^2
		\qquad\forall\alpha\in\set{1,2}^p,
		\\
		\label{eq:argmax_choice}
		\beta &\in 
		\argmax_{\alpha\in\set{1,2}^p}\; \norm{(\hat\mu^\alpha,\hat\nu^\alpha)}_2^2,
		\\
		\label{eq:hatmunu_choice}
		(\bar\mu,\bar\nu) &:= (\hat\mu^{\beta},\hat\nu^{\beta})
		\in\R^{2p}.
	\end{align}
	These choices are possible because
	the sets $B\cap A^\alpha$ is compact and nonempty
	(the nonemptiness follows from 
	$(\mu^\alpha,\nu^\alpha)\in B\cap A^\alpha$)
	and the set $\set{1,2}^p$ is finite.
	Furthermore, we have $(\bar\mu,\bar\nu)\in B$,
	i.e.\ it is a convex combination 
	as claimed.

	It remains to show that our choice $(\hat\mu^\beta,\hat\nu^\beta)$
	for $(\bar\mu,\bar\nu)$ satisfies \eqref{eq:m_stat}.
	By contradiction, we assume that there exists 
	an $i\in\hat I$ such that \eqref{eq:m_stat} is not satisfied,
	i.e.\ $\hat\mu^\beta_i\neq0$, $\hat\nu^\beta_i\neq0$,
	and $\hat\mu^\beta_i<0\lor\hat\nu^\beta_i<0$ hold.
	Without loss of generality we can assume that $\beta_i=1$ holds
	(otherwise one would exchange the roles of 
	$\mu$ and $\nu$ in the rest of the proof).
	Therefore, we have $\hat\mu^\beta_i\geq0$ due to 
	$(\hat\mu^\beta,\hat\nu^\beta)\in A^{\beta}$.
	Since \eqref{eq:m_stat} is not satisfied this implies
	$\hat\mu^\beta_i>0$ and $\hat\nu^\beta_i<0$.
	We define 
	\begin{equation*}
		\gamma\in\set{1,2}^p,
		\qquad
		\gamma_j:=
		\begin{cases}
			2 &\;\text{if}\;j=i,
			\\
			\beta_j &\;\text{if}\;j\in I^p\setminus\set{i}
		\end{cases}
		\quad\forall j\in I^p.
	\end{equation*}
	Due to $\hat\mu^\beta_i>0$ we can choose $t\in (0,1)$ such that
	the convex combination
	\begin{equation*}
		(\mu^t,\nu^t):=t(\hat\mu^{\gamma},\hat\nu^{\gamma})
		+ (1-t)(\hat\mu^\beta,\hat\nu^\beta)\in\R^{2p}
	\end{equation*}
	still satisfies $\mu^t_i>0$.
	Since $\gamma_j=\beta_j$ holds for $j\neq i$
	we also have $(\mu^t,\nu^t)\in A^{\beta}\cap B$.
	However,
	$\hat\nu^\gamma_i\geq0$ implies
	$(\hat\mu^\beta,\hat\nu^\beta)\neq(\hat\mu^{\gamma},\hat\nu^{\gamma})$.
	Thus, by also using \eqref{eq:argmax_choice}, we have
	\begin{equation*}
		\norm{(\mu^t,\nu^t)}_2^2
		< \max\set[\Big]{\norm{(\hat\mu^\beta,\hat\nu^\beta)}_2^2,
		\norm{(\hat\mu^{\gamma},\hat\nu^{\gamma})}_2^2}
		=\norm{(\hat\mu^{\beta},\hat\nu^{\beta})}_2^2.
	\end{equation*}
	Due to
	$(\mu^t,\nu^t)\in B\cap A^{\beta}$
	this is a contradiction to \eqref{eq:argmin_choice},
	which completes the proof.
\end{proof}
Instead of the function $(\mu,\nu)\mapsto \norm{(\mu,\nu)}_2^2$,
any strictly convex function on $\R^{2p}$ would have worked
in the proof.

We mention that it was recognized already in 
\cite[Section~4.4.2]{Schinabeck2009:1}
that this \lcnamecref{lem:schinabeck} would significantly
simplify the already existing proofs for M-stationarity.
%
We further mention that the idea to combine
various multipliers of the form $(\mu^\alpha,\nu^\alpha)$
was also used for the concept of $\QQ$-stationarity
in \cite{BenkoGfrerer2016:2}.

A straightforward combination of \cref{lem:astat,lem:schinabeck}
yields the desired M-stationarity result.
\begin{theorem}
	\label{thm:mstat}
	Let $\bar x\in\R^n$ be a local minimizer of \eqref{eq:mpcc}
	that satisfies MPCC-GCQ.
	Then $\bar x$ is an M-stationary point.
\end{theorem}
\begin{proof}
	For all $\alpha\in\set{1,2}^p$, let
	$(\lambda^\alpha,\eta^\alpha,\mu^\alpha,\nu^\alpha)\in\R^{l+m+2p}$
	be the multipliers generated by \cref{lem:astat}.
	By applying \cref{lem:schinabeck} with
	$\hat I=I^{00}(\bar x)$, we find a convex combination
	$(\bar\lambda,\bar\eta,\bar\mu,\bar\nu)\in\R^{l+m+2p}$
	of these multipliers such that \eqref{eq:mstat_00} is satisfied.
%
%
	The conditions \eqref{eq:wstat_x}--\eqref{eq:wstat_0+}
	follow from \eqref{eq:astat_x}--\eqref{eq:astat_0+} by convexity.
\end{proof}

\section{Conclusion and outlook}
\label{sec:conclusion}

We provided a new proof for the M-stationarity
of local minimizers of MPCCs under MPCC-GCQ.
Although this result was already known,
the new proof uses only basic and well-known tools
from nonlinear programming theory.
This new elementary proof for M-stationarity
was enabled by proving
a (to the best of our knowledge previously open) conjecture
from \cite{Schinabeck2009:1} in \cref{lem:schinabeck}.

In the future, it would also be interesting
to apply this approach to other problem classes
from disjunctive programming 
and to investigate to what extend the ideas
from this paper can be generalized.

In Sobolev or Lebesgue spaces,
the limiting normal cone turned out to be not as effective as in finite
dimensional spaces for obtaining stationarity conditions
for complementarity-type optimization problems, 
see \cite{HarderWachsmuth2017:1,MehlitzWachsmuth2016:1}.
Thus, it would be interesting to know whether
the new elementary method from this paper can provide ideas for
possible approaches for better stationarity conditions 
of complementarity-type optimization problems 
in Sobolev and Lebesgue spaces.

%

\bibliographystyle{jnsao}
\bibliography{m_stat_elementary_proof.bib}
\end{document}